\documentclass[12pt]{amsproc}
\usepackage{hyperref}
\usepackage{geometry}
\usepackage{graphicx}
\usepackage{amsthm}
\usepackage{amssymb}

\theoremstyle{plain}
\newtheorem{thm}{Theorem}[section]

\newtheorem{prop}[thm]{Proposition}
\newtheorem{defn}[thm]{Definition}

\newtheorem{rem}[thm]{Remark}

\begin{document}

\title{On Generalizations of Graded $r$-ideals}

\author{Rashid \textsc{Abu-Dawwas}}
\address{Department of Mathematics, Yarmouk University, Irbid, Jordan}
\email{rrashid@yu.edu.jo}

\author{Malik \textsc{Bataineh}}
\address{Department of Mathematics, Jordan University of Science and Technology, Irbid, Jordan}
\email{msbataineh@just.edu.jo}

\author{Ghida'a \textsc{Al-Qura'an}}
\address{Department of Mathematics, Yarmouk University, Irbid, Jordan}
\email{ghidaanaser.96@yahoo.com}



\begin{abstract}
In this article, we introduce a generalization of the concept of graded $r$-ideals in graded commutative rings with nonzero unity. Let $G$ be a group, $R$ be a $G$-graded commutative ring with nonzero unity and $GI(R)$ be the set of all graded ideals of $R$. Suppose that $\phi: GI(R)\rightarrow GI(R)\bigcup\{\emptyset\}$ is a function. A proper graded ideal $P$ of $R$ is called a graded $\phi$-$r$-ideal of $R$ if whenever $x, y$ are homogeneous elements of $R$ such that $xy\in P-\phi(P)$ and $Ann(x) =\{0\}$, then $y\in P$. Several properties of graded $\phi$-$r$-ideals have been examined.
\end{abstract}

\maketitle

\section{Introduction}

Throughout this article, $G$ will be a group with identity $e$ and $R$ a commutative ring with nonzero unity $1$. Then $R$ is called $G$-graded if $R=\displaystyle\bigoplus_{g\in G} R_{g}$ with $R_{g}R_{h}\subseteq R_{gh}$ for all $g, h\in G$ where $R_{g}$ is an additive subgroup of $R$ for all $g\in G$. The elements of $R_{g}$ are called homogeneous of degree $g$. If $a\in R$, then $a$ can be written uniquely as $\displaystyle\sum_{g\in G}a_{g}$, where $a_{g}$ is the component of $a$ in $R_{g}$. The component $R_{e}$ is a subring of $R$ and $1\in R_{e}$. The set of all homogeneous elements of $R$ is $h(R)=\displaystyle\bigcup_{g\in G}R_{g}$. Let $P$ be an ideal of a graded ring $R$. Then $P$ is called a graded ideal if $P=\displaystyle\bigoplus_{g\in G}(P\cap R_{g})$, i.e., for $a\in P$, $a=\displaystyle\sum_{g\in G}a_{g}$ where $a_{g}\in P$ for all $g\in G$. It is not necessary that every ideal of a graded ring is a graded ideal . For more details and terminology, see \cite{Hazart, Nastasescue}.

Let $I$ be a proper graded ideal of $R$. Then the graded radical of $I$ is $Grad(I)$, and is defined to be the set of all $x\in R$ such that for each $g\in G$, there exists a positive integer $n_{g}$ for which $x_{g}^{n_{g}}\in I$. One can see that if $x\in h(R)$, then $x\in Grad(I)$ if and only if $x^{n}\in I$ for some positive integer $n$. In fact, $Grad(I)$ is a graded ideal of $R$, see \cite{Refai Hailat}.

A proper graded ideal $P$ of $R$ is said to be graded prime if whenever $a, b\in h(R)$ such that $ab\in P$, then either $a\in P$ or $b\in P$ (\cite{Refai Hailat}). In 2006, Atani introduced in \cite{Atani Prime} the concept of graded weakly prime ideals, a proper graded ideal $P$ of $R$ is said to be a graded weakly prime ideal of $R$ if whenever $a, b\in h(R)$ such that $0\neq ab\in P$, then $a\in P$ or $b\in P$. In 2008, Jaber, Bataineh and Khashan in \cite{Jaber Bataineh Khashan} recently defined a proper graded ideal $P$ of $R$ is said to be graded almost prime if for $a, b\in h(R)$ such that $ab\in P-P^{2}$, then either $a\in P$ or $b\in P$. Also, graded almost prime ideals were generalized in \cite{Jaber Bataineh Khashan} to graded $n$-almost prime as follows: for $a, b\in h(R)$ such that $ab\in P-P^{n}$, then either $a\in P$ or $b\in P$. Later, in \cite{Alshehry, Alshehry 2}, a concept which covers all the previous definitions in a graded commutative ring $R$ has been introduced as follows: suppose that $\phi: GI(R)\rightarrow GI(R)\bigcup\{\emptyset\}$ is a function, where $GI(R)$ denotes the set of all graded ideals of $R$. A proper graded ideal $P$ of $R$ is said to be a graded $\phi$-prime ideal of $R$ if whenever $a, b\in h(R)$ such that $ab\in P-\phi(P)$, then $a\in P$ or $b\in P$. The following functions have been used:

\begin{enumerate}
\item $\phi_{\emptyset}(P)=\emptyset$ (graded prime ideal)

\item $\phi_{0}(P)=\{0\}$ (graded weakly prime ideal)

\item $\phi_{1}(P)=P$ (any graded ideal)

\item $\phi_{2}(P)=P^{2}$ (graded almost prime ideal)

\item $\phi_{n}(P)=P^{n}$ (graded almost $n$-prime ideal)

\item $\phi_{\omega}(P)=\displaystyle\bigcap_{n=1}^{\infty}P^{n}$ (graded $\omega$-prime ideal)
\end{enumerate}

In 2019, the concept of graded $r$-ideals has been introduced in \cite{Dawwas Bataineh}, a proper graded ideal $P$ of $R$ is said to be a graded $r$-ideal (resp. graded $pr$-ideal) of $R$ if whenever $x, y\in h(R)$ such that $xy\in P$ and $Ann(x)=\{0\}$, then $y\in P$ (resp. $y^{n}\in P$ for some positive integer $n$). Motivated by \cite{Ugurlu}, our goal is to introduce a generalization of the concept of graded $r$-ideals, a proper graded ideal $P$ of $R$ is called a graded $\phi$-$r$-ideal (resp. graded $\phi$-$pr$-ideal) of $R$ if whenever $x, y\in h(R)$ such that $xy\in P-\phi(P)$ and $Ann(x) =\{0\}$, then $y\in P$ (resp. $y^{n}\in P$ for some positive integer $n$). Among several results, we prove that if $P$ is a graded $\phi$-$r$-ideal of $R$ and $\phi(P)\subseteq r(R)=\left\{a\in R:Ann(a)=\{0\}\right\}$, then $P/\phi(P)$ is a graded $r$-ideal of $R/\phi(P)$, also, if $P/\phi(P)$ is a graded $r$-ideal of $R/\phi(P)$ and $\phi(P)$ is a graded $r$-ideal of $R$, then $P$ is a graded $\phi$-$r$-ideal of $R$ (Theorem \ref{Theorem 1}). We show that if $\phi(P)$ is a graded $r$-ideal of $R$, then $P$ is a graded $\phi$-$r$-ideal of $R$ if and only if $P$ is a graded $r$-ideal of $R$. (Theorem \ref{Theorem 3}). We prove that if $Grad(\phi(P))=\phi(Grad(P))$ and $P$ is a graded $\phi$-$r$-ideal of $R$, then $Grad(P)$ is a graded $\phi$-$r$-ideal of $R$ (Theorem \ref{Theorem 4}). We show that if $P$ is a graded $\phi$-$r$-ideal of $R$ such that $\phi(P)$ is a graded $r$-ideal of $R$, $I$ and $J$ are graded ideals of $R$ such that $IJ\subseteq P$, $IJ\nsubseteq \phi(P)$, $Ann(I) = \{0\}$ and $Ann(I)=Ann(c)$ for some $c\in I\bigcap h(R)$, then $J\subseteq P$ (Theorem \ref{Theorem 6}). We prove that if $P\subseteq r(R)$ is a graded ideal of $R$ contained in a graded ideal $I$ of $R$ and $I$ is a graded $\phi$-$r$-ideal of $R$, then $I/P$ is a graded $\phi_{P}$-$r$-ideal of $R/P$ (Theorem \ref{Theorem 7}). In Theorem \ref{Theorem 8}, we examine graded $\phi$-$r$-ideals in idealization rings. In Theorem \ref{Theorem 9}, we investigate graded $\phi$-$r$-ideals in multiplicative sets.

\section{Graded $\phi$-$r$-Ideals}

In this section, we introduce and study the concept of graded $\phi$-$r$-ideals.

\begin{defn}Let $R$ be a graded ring and $\phi: GI(R)\rightarrow GI(R)\bigcup\{\emptyset\}$ be a function. A proper graded ideal $P$ of $R$ is called a graded $\phi$-$r$-ideal (resp. graded $\phi$-$pr$-ideal) of $R$ if whenever $x, y\in h(R)$ such that $xy\in P-\phi(P)$ and $Ann(x) =\{0\}$, then $y\in P$ (resp. $y^{n}\in P$ for some positive integer $n$).
\end{defn}

\begin{rem}\label{1}
Consider the following functions:
\begin{enumerate}
\item $\phi_{\emptyset}(P)=\emptyset$ (graded $r$-ideal) (resp. graded $pr$-ideal)

\item $\phi_{0}(P)=\{0\}$ (graded weakly $r$-ideal) (resp. graded weakly $pr$-ideal)

\item $\phi_{1}(P)=P$ (any graded ideal)

\item $\phi_{2}(P)=P^{2}$ (graded almost $r$-ideal) (resp. graded almost $pr$-ideal)

\item $\phi_{n}(P)=P^{n}$ (graded almost $n$-$r$-ideal) (resp. graded almost $n$-$pr$-ideal)

\item $\phi_{\omega}(P)=\displaystyle\bigcap_{n=1}^{\infty}P^{n}$ (graded $\omega$-$r$-ideal) (resp. graded $\omega$-$pr$-ideal)
\end{enumerate}
\end{rem}

Since $P-\phi(P) = P-(P\bigcap\phi(P))$ for any graded ideal $P$, without generalization loss, throughout this article, we suppose it is $\phi(P)\subseteq P$. Moreover, For functions $\phi, \varphi:GI(R)\rightarrow GI(R)\bigcup\{\emptyset\}$, we write $\phi\leq\varphi$ if $\phi(P)\subseteq\varphi(P)$ for all $P\in GI(R)$. Obviously, therefore, we have the next order:
\begin{center}
$\phi_{\emptyset}\leq\phi_{0}\leq\phi_{\omega}\leq...\leq\phi_{n+1}\leq\phi_{n}\leq...\leq\phi_{2}\leq\phi_{1}$.
\end{center}

\begin{prop}\label{Proposition 1} Let $R$ be a graded ring and $P$ be a proper graded ideal of $R$. Assume that $\phi_{1}, \phi_{2}:GI(R)\rightarrow GI(R)\bigcup\{\emptyset\}$ be two functions with $\phi_{1}\leq\phi_{2}$.
\begin{enumerate}
\item If $P$ is a graded $\phi_{1}$-$r$-ideal of $R$, then $P$ is a graded $\phi_{2}$-$r$-ideal of $R$.

\item $P$ is a graded $r$-ideal if and only if $P$ is a graded weakly $r$-ideal. Also, if $P$ is a graded weakly $r$-ideal, then $P$ is a graded $\omega$-$r$-ideal, and then $P$ is a graded $(n+1)$-almost $r$-ideal, and then $P$ is a graded $n$-$r$-ideal ($n\geq2$), and then $P$ is a graded almost $r$-ideal.

\item $P$ is a graded $\omega$-$r$-ideal if and only if $P$ is a graded $n$-almost $r$-ideal for each $n\geq2$.

\item If $P$ is a graded idempotent ideal of $R$, then $P$ is a graded $\phi_{n}$-$r$-ideal of $R$ for every $n\geq1$.
\end{enumerate}
\end{prop}

\begin{proof}
\begin{enumerate}
\item It is straightforward.

\item It follows by (1).

\item Suppose that $P$ is a graded $\omega$-$r$-ideal. Since $\phi_{\omega}\leq\phi_{n}$, $P$ is a graded $n$-almost $r$-ideal. Conversely, let $x, y\in h(R)$ such that $xy\in P-\displaystyle\bigcap_{n=1}^{\infty}P^{n}$ and $Ann(x)=\{0\}$. Then $xy\in P-P^{n}$ for some $n\geq2$, and then since $P$ is a graded $n$-almost $r$-ideal of $R$, $y\in P$. Therefore, $P$ is a graded $\omega$-$r$-ideal.

\item The result follows since $P^{n}=P$ for all $n\geq2$.
\end{enumerate}
\end{proof}

Let $R$ be a graded ring and $P$ be a graded ideal of $R$. Then $R/P$ is a graded ring by $(R/P)_{g}=(R_{g}+P)/P$ for all $g\in G$ (\cite{Nastasescue}). Let $r(R)=\left\{a\in R:Ann(a)=\{0\}\right\}$.

\begin{thm}\label{Theorem 1} Let $R$ be a graded ring and $P$ be a proper graded ideal of $R$.
\begin{enumerate}
\item If $P$ is a graded $\phi$-$r$-ideal of $R$ and $\phi(P)\subseteq r(R)$, then $P/\phi(P)$ is a graded $r$-ideal of $R/\phi(P)$.

\item If $P/\phi(P)$ is a graded $r$-ideal of $R/\phi(P)$ and $\phi(P)$ is a graded $r$-ideal of $R$, then $P$ is a graded $\phi$-$r$-ideal of $R$.
\end{enumerate}
\end{thm}

\begin{proof}
\begin{enumerate}
\item By (\cite{Saber}, Lemma 3.2), $P/\phi(P)$ is a graded ideal of $R/\phi(P)$. Let $x+\phi(P), y+\phi(P)\in h(R/\phi(P))$ such that $(x + \phi(P))(y + \phi(P))\in P/\phi(P)$ and $Ann(x + \phi(P))=\{0+\phi(P)\}$. Then $x, y\in h(R)$ such that $xy\in P-\phi(P)$ and $Ann(x)=\{0\}$. Indeed, if $Ann(x)\neq\{0\}$, then there exists $0\neq a\in R$ such that $ax=0$, and then $a+\phi(P)\in R/\phi(P)$ with $(a+\phi(P))(x+\phi(P))=0+\phi(P)$, but $Ann(x+\phi(P))=\{0+\phi(P)\}$ which means that $a+\phi(P)=0+\phi(P)$, so $a\in \phi(P)\subseteq r(R)$ and then $Ann(a)=\{0\}$, which is a contradiction since $0\neq x\in Ann(a)$. Thus as $P$ is a graded $\phi$-$r$-ideal of $R$, we obtain that $y\in P$. This implies that $y + \phi(P)\in P/\phi(P)$.

\item Let $x, y\in h(R)$ such that $xy\in P-\phi(P)$ and $Ann(x)=\{0\}$. Then $x+\phi(P), y+\phi(P)\in h(R/\phi(P))$ such that $(x + \phi(P))(y + \phi(P))\in P/\phi(P)$ and $Ann(x+\phi(P))=\{0+\phi(P)\}$. Indeed, if $Ann(x + \phi(P))\neq\{0+\phi(P)\}$, then there exists $0\neq a + \phi(P)\in R/\phi(P)$ such that $(x + \phi(P))(a + \phi(P)) = 0 + \phi(P)$. Then $xa\in\phi(P)$. Since $\phi(P)$ is a graded $r$-ideal, we get $a\in\phi(P)$, which is a contradiction. Therefore, since $P/\phi(P)$ is a graded $r$-ideal, we have $y+\phi(P)\in P/\phi(P)$, so $y\in P$.
\end{enumerate}
\end{proof}

\begin{thm}\label{Theorem 2} Let $R$ be a graded ring and $P$ be a proper graded ideal of $R$. Then the followings statements are equivalent:
\begin{enumerate}
\item $P$ is a graded $\phi$-$r$-ideal of $R$.

\item For every $a\in h(R)\bigcap r(R)$, $(P : a) = P\bigcup (\phi(P) : a)$.

\item For every $a\in h(R)\bigcap r(R)$, $(P : a) = P$ or $(P : a) = (\phi(P) : a)$.
\end{enumerate}
\end{thm}

\begin{proof}$(1)\Rightarrow(2)$: Let $a\in h(R)\bigcap r(R)$ and $b\in(P:a)$. Then $b_{g}\in (P:a)$ for all $g\in G$ as $(P:a)$ is a graded ideal of $R$ by (\cite{Dawwas Yildiz}, Lemma 2), and then for any $g\in G$, $ab_{g}\in P$. If $ab_{g}\in \phi(P)$, then $b_{g}\in(\phi(P):a)$ for all $g\in G$, and then $b\in(\phi(P):a)$. If $ab_{g}\notin\phi(P)$, then as $P$ is a graded $\phi$-$r$-ideal and $a\in r(R)$, $b_{g}\in P$ for all $g\in G$, and then $b\in P$. Consequently, $b\in \left(P\bigcup(\phi(P):a)\right)$. The other containment is clear.

$(2)\Rightarrow(3)$: It is clear.

$(3)\Rightarrow(1)$: Let $b, c\in h(R)$ such that $bc\in P-\phi(P)$ and $b\in r(R)$. By (3), we conclude that $(P : b) = P$ or $(P : b) = (\phi(P) : b)$.
Suppose that $(P : b) = P$. Then as $bc\in P$, we get $c\in(P : b) = P$, as needed. Suppose that $(P : b) = (\phi(P) : b)$. Since $c\in(P : b)$, we say $bc\in\phi(P)$, which is a contradiction.
\end{proof}

\begin{prop}\label{Proposition 2} Let $R$ be a graded ring and $P$ be a proper graded ideal of $R$. If $P$ is a graded $\phi$-$r$-ideal of $R$, then $\left(P\bigcap h(R)\right)-\phi(P)\subseteq Zd(R)$, where $Zd(R)$ is the set of all zero divisors of $R$.
\end{prop}

\begin{proof}Suppose that $\left(P\bigcap h(R)\right)-\phi(P)\nsubseteq Zd(R)$. Then there exists $a\in \left(P\bigcap h(R)\right)-\phi(P)$ with $a\in r(R)$, and then $a, 1\in h(R)$ such that $a.1\in P-\phi(P)$. So, as $P$ is a graded $\phi$-$r$-ideal, we get $1\in P$, which is a contradiction.
\end{proof}

\begin{thm}\label{Theorem 3} Let $R$ be a graded ring and $P$ be a proper graded ideal of $R$ such that $\phi(P)$ is a graded $r$-ideal of $R$. Then $P$ is a graded $\phi$-$r$-ideal of $R$ if and only if $P$ is a graded $r$-ideal of $R$.
\end{thm}

\begin{proof}Suppose that $P$ is a graded $\phi$-$r$-ideal of $R$. Let $x, y\in h(R)$ such that $xy\in P$ and $x\in r(R)$. If $xy\notin \phi(P)$, then since $P$ is a graded $\phi$-$r$-ideal, we have $y\in P$, as desired. If $xy\in \phi(P)$, then as $\phi(P)$ is a graded $r$-ideal, we conclude that $y\in \phi(P)\subseteq P$. The converse is clear.
\end{proof}

\begin{thm}\label{Theorem 4} Let $R$ be a graded ring and $P$ be a proper graded ideal of $R$ with $Grad(\phi(P))=\phi(Grad(P))$. If $P$ is a
graded $\phi$-$r$-ideal of $R$, then $Grad(P)$ is a graded $\phi$-$r$-ideal of $R$.
\end{thm}

\begin{proof}Let $x, y\in h(R)$ such that $xy\in Grad(P)-\phi(Grad(P))$ and $x\in r(R)$. As $xy\in Grad(P)$, there exists a positive integer $k$ such that $(xy)^{k} = x^{k}y^{k}\in P$. On the other hand, $xy\notin \phi(Grad(P)) = Grad(\phi(P))$, one can say that $x^{k}y^{k}\notin \phi(P)$. Also, since $x\in r(R)$, it is obvious that $x^{k}\in r(R)$. As $P$ is a graded $\phi$-$r$-ideal, we get $y^{k}\in P$, i.e., $y\in Grad(P)$.
\end{proof}

\begin{prop}\label{Proposition 3} Let $R$ be a graded ring and $P$ be a proper graded ideal of $R$ such that $\phi(P)$ is a graded $r$-ideal of $R$. If $P$ is a graded $\phi$-$r$-ideal of $R$, then $P\bigcap h(R)\subseteq Zd(R)$.
\end{prop}

\begin{proof}Suppose that $P\bigcap h(R)\nsubseteq Zd(R)$. Then there is $a\in P\bigcap h(R)$ but $a\notin Zd(R)$. So, $a, 1\in h(R)$ such that $a.1\in P$. If $a\notin \phi(P)$, then $1\in P$ since $P$ is a graded $\phi$-$r$-ideal, which is a contradiction. If $a\in \phi(P)$, then we obtain $1\in \phi(P)\subseteq P$ since $\phi(P)$ is a graded $r$-ideal, which is a contradiction.
\end{proof}

\begin{prop}\label{Proposition 4} Let $R$ be a graded ring and $P$ be a graded prime ideal of $R$ such that $\phi(P)$ is a graded $r$-ideal of $R$. Then
$P$ is a graded $\phi$-$r$-ideal if and only if $P\bigcap h(R)\subseteq Zd(R)$.
\end{prop}

\begin{proof}Suppose that $P\bigcap h(R)\subseteq Zd(R)$. Let $x, y\in h(R)$ such that $xy\in P-\phi(P)$ and $x\in r(R)$. Since $P$ is graded prime, either $x\in P$ or $y\in P$. If $x\in P$, then $x\in Zd(R)$ by assumption, which is a contradiction since $x\in r(R)$. Consequently, $y\in P$, as desired. The converse is clear by Proposition \ref{Proposition 3}.
\end{proof}

\begin{prop}\label{Proposition 5}Let $R$ be a graded ring, $P$ be a proper graded ideal of $R$ and $a\in h(R)-P$. Suppose that $(\phi(P) : a)\subseteq \phi((P : a))$. If $P$ is a graded $\phi$-$r$-ideal of $R$, then $(P : a)$ is a graded $\phi$-$r$-ideal of $R$.
\end{prop}

\begin{proof}By (\cite{Dawwas Yildiz}, Lemma 2), $(P:a)$ is a graded ideal of $R$. Let $b, c\in h(R)$ such that $bc\in (P : a)-\phi((P : a))$ and $b\in r(R)$. Then we get $bca\in P$ and $bca\notin \phi(P)$ by $(\phi(P) : a)\subseteq \phi((P : a))$. Since $b\in r(R)$ and $P$ is a graded $\phi$-$r$-ideal, we get $ca\in P$, i.e., $c\in (P : a)$, as needed.
\end{proof}

\begin{defn}Let $R$ be a graded ring and $P$ be a proper graded ideal of $R$. Then $P$ is said to be a graded strongly $\phi$-$r$-ideal of $R$ if for every graded ideals $I$ and $J$ of $R$ such that $IJ\subseteq P$, $IJ\nsubseteq \phi(P)$ and $Ann(I) = \{0\}$, we have $J\subseteq P$.
\end{defn}

\begin{prop}\label{Proposition 6} Let $R$ be a graded ring. Then every graded strongly $\phi$-$r$-ideal of $R$ is a graded $\phi$-$r$-ideal.
\end{prop}

\begin{proof}Suppose that $P$ is a graded strongly $\phi$-$r$-ideal of $R$. Let $a, b\in h(R)$ such that $ab\in P-\phi(P)$ and $a\in r(R)$. Then $I=\langle a\rangle$ and $J=\langle b\rangle$ are graded ideals of $R$ by (\cite{Farzalipour}, Lemma 2.1), with $IJ\subseteq P$ and $IJ\nsubseteq \phi(P)$. Observe that $Ann(I) = \{0\}$. Suppose that for $0\neq u\in R$, $uI= \{0\}$. This means that for all $s\in R$, $usa = 0$. Since $a\in r(R)$, we obtain $us = 0$ for all $s\in R$. This implies $u = 0$, a contradiction. Thus as $P$ is graded strongly $\phi$-$r$-ideal, $J\subseteq P$, so $b\in P$.
\end{proof}

\begin{thm}\label{Theorem 6} Let $R$ be a graded ring and $P$ be a graded $\phi$ $r$-ideal of $R$. Suppose that $\phi(P)$ is a graded $r$-ideal of $R$. Assume that $I$ and $J$ are graded ideals of $R$ such that $IJ\subseteq P$, $IJ\nsubseteq \phi(P)$ and $Ann(I) = \{0\}$. If $Ann(I)=Ann(c)$ for some $c\in I\bigcap h(R)$, then $J\subseteq P$.
\end{thm}

\begin{proof}Suppose that $J\nsubseteq P$. Then there is $0 \neq b\in J$ with $b\notin P$, and then there is $g\in G$ such that $b_{g}\notin P$. Note that, $b_{g}\in J$ as $J$ is a graded ideal. On the other hand, since $Ann(I) = Ann(c)$, $c\in r(R)$, as $Ann(I) = \{0\}$. Consider $cb_{g}\in P$. If $cb_{g}\notin \phi(P)$, then since $P$ is a graded $\phi$-$r$-ideal, we obtain $b_{g}\in P$, a contradiction. If $cb_{g}\in \phi(P)$, then as $\phi(P)$ is a graded $r$-ideal, again we conclude $b_{g}\in P$, a contradiction. Thus it must be $J\subseteq P$.
\end{proof}

Let $R$ be a graded ring and $P$ be a graded ideal of $R$. Then $R/P$ is a graded ring by $(R/P)_{g}=(R_{g}+P)/P$ for all $g\in G$ (\cite{Nastasescue}). Define $\phi_{P} : GI(R/P)\rightarrow GI(R/P)\bigcup\{\emptyset\}$ by $\phi_{P} (I/P) =(\phi(I) + P)/P$ for every graded ideal $I$ of $R$ with $P\subseteq I$ and $\phi_{P} (I/P) =\emptyset$ if $\phi(I) =\emptyset$. Notice that $\phi_{P} (I/P)\subseteq I/P$.

\begin{thm}\label{Theorem 7} Let $R$ be a graded ring and $P\subseteq r(R)$ be a graded ideal of $R$ contained in a graded ideal $I$ of $R$. If $I$ is a
graded $\phi$-$r$-ideal of $R$, then $I/P$ is a graded $\phi_{P}$-$r$-ideal of $R/P$.
\end{thm}

\begin{proof}By (\cite{Saber}, Lemma 3.2), $I/P$ is a graded ideal of $R/P$. Let $x + P, y + P\in h(R/P)$ such that $(x + P)(y + P)\in I/P-\phi_{P} (I/P)$ and $x + P\in r(R/P)$. Then $x, y\in h(R)$ such that $xy\in I-\phi(I)$. Also, since $x + P\in r(R/P)$, $x\in r(R)$. Indeed, if $x\notin r(R)$, there is
$0\neq a\in R$ such that $ax = 0$. Then we have $(x+P)(a +P) = 0+P$. Since $x+P\in r(R/P)$, $a + P = 0 + P$, i.e., $a\in P\subseteq r(R)$. This gives us a contradiction as $0\neq x\in Ann(a)$. Thus, as $I$ is a graded $\phi$-$r$-ideal, $y\in I$, so $y + P\in I/P$.
\end{proof}

\begin{prop}\label{Proposition 7} Let $R$ be a graded ring and $P$ be a graded $r$-ideal of $R$ contained in a graded ideal $I$ of $R$. If $I/P$ is a graded $r$-ideal of $R/P$, then $I$ is a graded $\phi$-$r$-ideal of $R$.
\end{prop}

\begin{proof}Let $a, b\in h(R)$ such that $ab\in I-\phi(I)$ and $a\in r(R)$. If $ab\in P$, then as $P$ is a graded $r$-ideal, we get $b\in P\subseteq I$, as desired. Suppose that $ab\in I-P$. Then $a+P, b+P\in h(R/P)$ such that $ab + P = (a + P)(b + P)\in I/P$. Also, as $a\in r(R)$, $a + P\in r(R/P)$. Indeed, if $a + P\notin r(R/P)$, then there is $0\neq x+ P\in R/P$ such that $(a +P)(x +P) = 0 +P$. This means that $ax\in P$. But as $a\in r(R)$ and $P$ is a graded $r$-ideal, we see $x\in P$, a contradiction. Therefore, since $I/P$ is a graded $r$-ideal of $R/P$ and $a + P\in r(R/P)$, we see $b + P\in I/P$, so $b\in I$.
\end{proof}

\begin{defn}Let $R$ be a graded ring and $\phi: GI(R)\rightarrow GI(R)\bigcup\{\emptyset\}$ be a function. A proper graded ideal $P$ of $R$ is called a graded $\phi$-pure ideal of $R$ if for every $x\in \left(P\bigcap h(R)\right)-\phi(P)$, there is $y\in P\bigcap h(R)$ such that $x = xy$.
\end{defn}

\begin{prop}A proper graded ideal $P$ of $R$ is a graded $\phi$-pure ideal if and only if for every nonzero $x\in \left(P\bigcap h(R)\right)-\phi(P)$, there is $y\in P_{e}$ such that $x = xy$.
\end{prop}

\begin{proof}Suppose that $P$ is a graded $\phi$-pure ideal of $R$. Let $x\in \left(P\bigcap h(R)\right)-\phi(P)$ be nonzero. Then there is $y\in P\bigcap h(R)$ such that $x = xy$. Now, as $x, y\in h(R)$, $x\in R_{g}$ and $y\in R_{h}$ for some $g, h\in G$, and then $x=xy\in R_{g}R_{h}\subseteq R_{gh}$. So, $0\neq x\in R_{g}\bigcap R_{gh}$ which implies that $gh=g$, and hence $h=e$. Thus, $y\in P\bigcap R_{e}=P_{e}$. Conversely, let $x\in \left(P\bigcap h(R)\right)-\phi(P)$. If $x=0$, then $y=0\in P\bigcap h(R)$ such that $x=xy$. Suppose that $x\neq0$. Then by assumption, there is $y\in P_{e}=P\bigcap R_{e}\subseteq P\bigcap h(R)$ such that $x=xy$. Hence, $P$ is a graded $\phi$-pure ideal of $R$.
\end{proof}

\begin{defn}Let $R$ be a graded ring and $\phi: GI(R)\rightarrow GI(R)\bigcup\{\emptyset\}$ be a function. A proper graded ideal $P$ of $R$ is called a graded $\phi$-von Neumann regular ideal of $R$ if for every $x\in \left(P\bigcap h(R)\right)-\phi(P)$, there is $y\in P\bigcap h(R)$ such that $x = x^{2}y$.
\end{defn}

\begin{prop}A proper graded ideal $P$ of $R$ is a graded $\phi$-von Neumann regular ideal of $R$ if and only if for every nonzero $x\in \left(P\bigcap h(R)\right)-\phi(P)$, there are $g\in G$ and $y\in P_{g^{-1}}$ with $x = x^{2}y$.
\end{prop}

\begin{proof}Suppose that $P$ is a graded $\phi$-von Neumann regular ideal of $R$. Let $x\in \left(P\bigcap h(R)\right)-\phi(P)$ be nonzero. Then there is $y\in P\bigcap h(R)$ such that $x = x^{2}y$. Now, as $x, y\in h(R)$, $x\in R_{g}$ and $y\in R_{h}$ for some $g, h\in G$, and then $x=x^{2}y\in R_{g}R_{g}R_{h}\subseteq R_{g^{2}h}$. So, $0\neq x\in R_{g}\bigcap R_{g^{2}h}$ which implies that $g^{2}h=g$, and hence $h=g^{-1}$. Thus, $y\in P\bigcap R_{g^{-1}}=P_{g^{-1}}$. Conversely, let $x\in \left(P\bigcap h(R)\right)-\phi(P)$. If $x=0$, then $y=0\in P\bigcap h(R)$ such that $x=xy$. Suppose that $x\neq0$. Then by assumption, there are $g\in G$ and $y\in P_{g^{-1}}=P\bigcap R_{g^{-1}}\subseteq P\bigcap h(R)$ such that $x=xy$. Hence, $P$ is a graded $\phi$-von Neumann regular ideal of $R$.
\end{proof}

\begin{rem}We define the concepts in Remark \ref{1} for graded $\phi$-pure ideal and graded $\phi$-von Neuman regular ideal.
\end{rem}

Assume that $M$ is a left $R$-module. Then $M$ is said to be $G$-graded if
$M=\displaystyle\bigoplus_{g\in G}M_{g}$ with $R_{g}M_{h}\subseteq M_{gh}$ for
all $g,h\in G$ where $M_{g}$ is an additive subgroup of $M$ for all $g\in G$.
The elements of $M_{g}$ are called homogeneous of degree $g$. It is clear that
$M_{g}$ is an $R_{e}$-submodule of $M$ for all $g\in G$. We assume that
$h(M)=\displaystyle\bigcup_{g\in G}M_{g}$. Let $N$ be an $R$-submodule of a
graded $R$-module $M$. Then $N$ is said to be graded $R$-submodule if
$N=\displaystyle\bigoplus_{g\in G}(N\cap M_{g})$, i.e., for $x\in N$,
$x=\displaystyle\sum_{g\in G}x_{g}$ where $x_{g}\in N$ for all $g\in G$. It is
known that an $R$-submodule of a graded $R$-module need not be graded.

Let $M$ be an $R$-module. The idealization $R(+)M=\left\{  (r,m):r\in
R\mbox{ and }m\in M\right\}  $ of $M$ is a commutative ring with componentwise
addition and multiplication; $(x,m_{1})+(y,m_{2})=(x+y,m_{1}+m_{2})$ and
$(x,m_{1})(y,m_{2})=(xy,xm_{2}+ym_{1})$ for each $x,y\in R$ and $m_{1}%
,m_{2}\in M$. Let $G$ be an abelian group and $M$ be a $G$-graded $R$-module.
Then $X=R(+)M$ is $G$-graded by $X_{g}=R_{g}(+)M_{g}$ for all $g\in G$. Note
that, $X_{g}$ is an additive subgroup of $X$ for all $g\in G$. Also, for
$g,h\in G$, $X_{g}X_{h}=(R_{g}(+)M_{g})(R_{h}(+)M_{h})=(R_{g}R_{h},R_{g}%
M_{h}+R_{h}M_{g})\subseteq(R_{gh},M_{gh}+M_{hg})\subseteq(R_{gh}%
,M_{gh})=X_{gh}$ as $G$ is abelian\ \cite{RaTeShKo}. If $P$ is
an ideal of $R$ and $N$ is an $R$-submodule of $M$ such that $PM\subseteq N$,
then $P(+)N$ is a graded ideal of $R(+)M$ if and only if $P$ is a graded ideal
of $R$ and $N$ is a graded $R$-submodule of $M$ (\cite{RaTeShKo}, Proposition 3.3). Note that
\begin{center}
$Zd(R(+)M) = \left\{(a,m) : a\in Zd(R)\bigcup Zd(M)\right\}$, where $Zd(M) = \left\{a\in R : am = 0 \mbox{ for some }0\neq m\in M\right\}$.
\end{center}
The authors in\cite{RaTeShKo} determined the certain classes of graded ideals such as graded
maximal ideal, graded prime ideals, graded primary ideals, graded quasi
primary ideals, graded 2-absorbing ideals and graded 2-absorbing quasi primary
ideals of graded idealization $R(+)M$. Now, we investigate the graded
$\phi$-$r$-ideals in $R(+)M$.

\begin{thm}\label{Theorem 8} Let $G$ be an abelian group and $M$ be a $G$-graded $R$-module such that $Zd(R) = Zd(M)$. Let $\phi_{1}: GI(R)\rightarrow GI(R)\bigcup\{\emptyset\}$ and $\phi_{2}: GI(R(+)M)\rightarrow GI(R(+)M)\bigcup\{\emptyset\}$ be two function such that $\phi_{2}(P(+)M) =  \phi_{1}(P)(+)M$ for a proper graded ideal $P$ of $R$. If $P(+)M$ is a graded $\phi_{2}$-$r$-ideal of $R(+)M$, then $P$ is a graded $\phi_{1}$-$r$-ideal of $R$.
\end{thm}

\begin{proof}Let $x, y\in h(R)$ such that $xy\in P-\phi_{1}(P)$ and $x\in r(R)$. Since $\phi_{2}(P(+)M) = \phi_{1}(P)(+)M$, we have $(x, 0)(y, 0)\in P(+)M-\phi_{2}(P(+)M)$. Also, since $Zd(R) =Zd(M)$, it is clear that $(x, 0)\in r(R(+)M)$. This means that $(y, 0)\in P(+)M$, so $y\in P$, as needed.
\end{proof}

Suppose that $R$ is a graded ring. Let $S\subseteq h(R)$ be a multiplicative set. Then $S^{-1}R$ is a graded ring with $(S^{-1}R)_{g}=\left\{\frac{a}{s}:a\in R_{h}, s\in S\cap R_{hg^{-1}}\right\}$. Assume that $\phi:GI(R)\rightarrow GI(R)\bigcup\{\emptyset\}$ is a function. Define $\phi_{S}:GI(S^{-1}R)\rightarrow GI(S^{-1}R)\bigcup\{\emptyset\}$ by $\phi_{S}(I)=S^{-1}\phi(I\bigcap R)$ for every graded ideal $I$ of $S^{-1}R$, and $\phi_{S}(I)=\emptyset$ if $\phi(I\bigcap R)=\emptyset$. Note that, $\phi_{S}(I)\subseteq I$.

\begin{thm}\label{Theorem 9} Let $R$ be a graded ring, $\phi: GI(R)\rightarrow GI(R)\bigcup\{\emptyset\}$ be a function and $S$ be a multiplicative subset of $h(R)$ such that $S\subseteq r(R)$. If $P$ is a graded $\phi$-$r$-ideal of $R$ such that $S\bigcap P =\emptyset$ and $S^{-1}\phi(P)\subseteq \phi_{S}(S^{-1}P)$, then $S^{-1}P$ is a graded $\phi_{S}$-$r$-ideal of $S^{-1}R$. Moreover, if $S^{-1}P\neq S^{-1}\phi(P)$, then $S^{-1}P\bigcap R\subseteq Zd(R)$.
\end{thm}

\begin{proof}Let $\frac{x}{s}, \frac{y}{t}\in h(S^{-1}R)$ such that $\frac{x}{s}\frac{y}{t}\in S^{-1}P-\phi_{S}(S^{-1}P)$ and $\frac{x}{s}\in r(S^{-1}R)$.
Then there is $s_{1}\in S$ such that $s_{1}xy\in P$. Also, it is clear that since $\frac{x}{s}\in r(S^{-1}R)$, $x\in r(R)$. On the other hand, $\frac{x}{s}\frac{y}{t}\notin \phi_{S}(S^{-1}P)$ implies that $s_{2}xy\notin \phi_{S}(S^{-1}P)\bigcap R$ for all $s_{2}\in S$. By $S^{-1}\phi(P)\subseteq \phi_{S}(S^{-1}P)$, we have $s_{2}xy\notin \phi(P)$ for all $s_{2}\in S$. Thus $xs_{1}y\in P-\phi(P)$. Then as $P$ is a graded $\phi$-$r$-ideal, $s_{1}y\in P$. Hence $\frac{y}{t} = \frac{s_{1}y}{s_{1}t}\in S^{-1}P$, as needed. Moreover, assume that $S^{-1}P\neq S^{-1}\phi(P)$. Let $x\in S^{-1}P\bigcap R$. Then $x_{g}\in S^{-1}P\bigcap R$ for all $g\in G$ as $S^{-1}P$ is a graded ideal, and then for $g\in G$, there is $s\in S$ such that $x_{g}s\in P$. Also, $s\notin P$ by $S\bigcap P = \emptyset$. If $x_{g}s\notin \phi(P)$, then we obtain $x_{g}\notin r(R)$ since $s\notin P$. If $x_{g}s\in \phi(P)$, then $x_{g}\in S^{-1}\phi(P)\bigcap R$. Thus $S^{-1}P\bigcap R\subseteq Zd(R)\bigcup\left(S^{-1}\phi(P)\bigcap R\right)$. Then by our assumption $S^{-1}P\neq S^{-1}\phi(P)$, we obtain $S^{-1}P\bigcap R\subseteq Zd(R)$.
\end{proof}

\end{document}